\documentclass[a4paper, 12pt]{article}
\usepackage{amsmath, amssymb, amsfonts, amsthm}
\usepackage{fancybox}
\usepackage{fullpage}

\newlength{\querylen}
\theoremstyle{plain}
\newtheorem{theorem}{Theorem}[section]
\newtheorem{lemma}[theorem]{Lemma}

\newtheorem{corollary}[theorem]{Corollary}

\theoremstyle{definition}
\newtheorem{definition}{Definition}[section]

\theoremstyle{remark}

\newtheorem{example}{Example}[section]

\newcommand{\sA}{\mathcal{A}}
\newcommand{\sB}{\mathcal{B}}
\newcommand{\sC}{\mathcal{C}}
\newcommand{\sF}{\mathcal{F}}

\newcommand{\B}{\mathbb{B}}

\newcommand{\K}{\mathbb{K}}

\newcommand{\R}{\mathbb{R}}
\renewcommand{\S}{\mathbb{S}}

\renewcommand{\d}{\mathrm{d}} 

\newcommand{\neutral}{\mathbf{e}}
\newcommand{\zero}{\mathbf{0}}
\newcommand{\eps}{\varepsilon}

\DeclareMathOperator{\interior}{int}
\DeclareMathOperator{\closure}{cl}

\newcommand{\I}[1]{\mathbf{1}_{#1}}
\renewcommand{\P}{\mathbf{P}}
\newcommand{\E}{\mathbf{E}}
\newcommand{\cE}[2]{\mathbf{E}(#1 \,|\, #2)}
\DeclareMathOperator{\Var}{Var}

\newcommand{\iid}{i.\,i.\,d.\ }

\newcommand{\norm}[1]{\|#1\|}

\begin{document}

 \title{Large deviations for heavy-tailed random elements in convex
  cones}

\author{Christoph Kopp\thanks{Partially supported by Swiss National
    Foundation Grant Nr. 200021-126503, e-mail:
    \texttt{chis.kopp@stat.unibe.ch}} \ and Ilya
    Molchanov\thanks{e-mail: \texttt{ilya.molchanov@stat.unibe.ch}}}

\date{\small{Institute of Mathematical Statistics and Actuarial Science,\\
  University of Bern, Sidlerstrasse 5, 3012 Bern, Switzerland}}

\maketitle
\begin{abstract}
  We prove large deviation results for sums of heavy-tailed random
  elements in rather general convex cones being semigroups equipped
  with a rescaling operation by positive real numbers. In difference
  to previous results for the cone of convex sets, our technique does
  not use the embedding of cones in linear spaces. Examples include
  the cone of convex sets with the Minkowski addition,
  positive half-line with maximum operation and the family of square
  integrable functions with arithmetic addition and argument
  rescaling.
\end{abstract}



\section{Introduction}

Most results concerning limiting behavior of sums of random elements
in linear spaces can be extended for random closed sets in linear
spaces, see \cite[Ch.~3]{mo1}. The sum of sets is defined in the
Minkowski sense, i.\,e. the sum of two sets is the closure of the set
of pairwise sums of elements from these sets. It is well known that
this addition is not invertible. The most typical way to handle this
setting is to consider first random convex compact sets and embed them
into the Banach space of continuous functions on the unit sphere in
the dual space using the support function. Then the Minkowski sum of
sets corresponds to the arithmetic sum of their support functions and
the Hausdorff distance between sets turns into the uniform distance in
the space of support functions, which opens the possibility to use the
results available for random elements in Banach spaces, see e.\,g.
\cite{gin:h:z}. Finally, it is usually argued that the results for
possibly non-convex random compact sets are identical to their convex
case counterparts in view of the convexification property of the
Minkowski addition, see \cite{art:han85}.

The family of limit theorems for random sets has been recently
extended with several large deviation results in the heavy-tail
setting in \cite{MPSa} and \cite{MPSb}.  The crucial assumption is the
regular variation condition on the tail, which is similar to one that
appears in limit theorems for unions of random closed sets, see
\cite[Ch.~4]{mo1}. Let $S_n$ denote the Minkowski sum of \iid
regularly varying random compact sets $\xi_1,\dots,\xi_n$ in $\R^m$
with tail index $\alpha>0$ and tail measure $\mu$. In particular,
\cite{MPSa} show that
\begin{equation}
  \label{eq:mps1}
  \gamma_n\P(S_n\in\lambda_n U) \to \mu(U)
\end{equation}
for $\lambda_n$ that grows sufficiently fast and all $\mu$-continuous
measurable subsets $U$ of the family of all compact sets bounded away
from zero. The sequence of normalising constants $\{\gamma_n\}$ is
related to the tail behavior of the norm of a single random compact
set defined as its Hausdorff distance to the origin. Especially, it is
required that $\lambda_n/n\to\infty$ in case $\alpha\geq 1$.

This large deviation result has been refined in \cite{MPSb}, where it
is shown that, for regularly varying convex random compact sets with
integrable norm (so $\alpha\geq 1$),
\begin{equation}
  \label{eq:mps2}
  \gamma_n \P(S_n\in\lambda_n U +\,n \E \xi_1)\to\mu(U)\,,
\end{equation}
where $\lambda_n$ grows slower than in \eqref{eq:mps1} and $\E\xi_1$
is the expectation of $\xi_1$, see \cite[Sec.~2.1]{mo1}.  The method
of the proof is based on the embedding argument combined with a use of
classical large deviation results from \cite{PVV} and \cite{hul:lin:mik:05}.

The setting of random compact sets can be considered as a special case
of random elements in convex cones (also called conlinear spaces),
being semigroups with a scaling operation by positive reals, see
\cite{DMZ}. A simple example is the cone of positive numbers with the
maximum operation. It should be noted that in that case the embedding
argument is not applicable any longer, so one has to prove the
corresponding results in the cone without using any centering or
symmetrization arguments.

In this paper, we generalize the above mentioned results from
\cite{MPSa} and \cite{MPSb} for heavy-tailed random elements in convex
cones. While the general scheme of our proofs follows the lines of the
proofs from \cite{MPSa} and \cite{MPSb}, it requires extra care caused
by the impossibility to use the embedding device. In particular, this
concerns our generalization of \eqref{eq:mps2}, since there is no
generally consistent definition of the expectation in convex cones.

\section{Regularly varying random elements in cones}
\label{sec:general}

A Borel function $f:(c,\infty)\mapsto(0,\infty)$ for some $c>0$ is
said to be \emph{regularly varying} (at infinity) with index $\rho$ if
\begin{displaymath}
  f(\lambda x)/f(x) \to \lambda^{\rho} \quad\text{as } x\to\infty
\end{displaymath}
for all $\lambda>0$, see e.\,g. \cite{BGT}.  If $\rho=0$, then $f$ is
called \emph{slowly varying} and usually denoted by the letter $\ell$
instead of $f$. Any regularly varying function $f$ with index $\rho$
has a representation $f(x) = x^{\rho}\ell(x)$ for a slowly varying
function $\ell$.  We write $f\sim g$ as a shorthand for $f(x)/g(x)\to
1$ as $x\to\infty$.

\begin{theorem}[Karamata, see Th.~1.5.11 \cite{BGT}]
  If $f$ is regularly varying with index $\rho$ and locally bounded on
  $[a,\infty)$, then
  \begin{enumerate}
  \item[(i)] for any $\beta\geq - (\rho+1),$
    \begin{displaymath}
      \lim_{x\to\infty}\frac{x^{\beta+1}f(x)}{\int_a^x t^{\beta}f(t)\d
        t} = \beta +\rho +1;
    \end{displaymath}
  \item[(ii)] for any $\beta < - (\rho+1)$ (and for $\beta=-(\rho+1)$ if
    $\int_a^{\infty} t^{-(\rho+1)}f(t)\,\d t <\infty$),
    \begin{displaymath}
      \lim_{x\to\infty}\frac{x^{\beta+1}f(x)}{\int_x^{\infty} t^{\beta}f(t)\d
        t} = -(\beta +\rho +1).
    \end{displaymath}
  \end{enumerate}
\end{theorem}

Below we summarize several concepts from \cite{DMZ} concerning general
convex cones.  A \emph{convex cone} $\K$ is a topological semigroup
with neutral element $\neutral$ and an extra operation $x\mapsto ax$
of scaling $x\in\K$ by a positive number $a$, so that $a(x+y) = ax +
ay$ for all $a>0,x,y\in\K$. It should be noted that we do \emph{not}
require the validity of the second distributivity law $(a+b)x = ax +
bx$. The second distributivity law holds for the cone of compact sets
in $\R^d$ with the Minkowski addition and enables using the embedding
argument.

We assume that $\K$ is a pointed cone, i.\,e. $ax$ converges to the
cone element $\zero$ called the origin as $a\downarrow0$ for all
$x\neq\neutral$. Assume that $\K$ is metrized by a homogeneous metric
$d$, i.\,e. $d(ax,ay)=ad(x,y)$ for all $x,y\in\K$ and $a>0$.  The
value $\norm{x}=d(x,\zero)$ is called the \emph{norm} of $x$ which in
general constitutes an abuse of language since $\norm{\cdot}$ is not
necessarily sub-linear. Nevertheless, the norm is sub-linear if the
metric is \emph{sub-invariant}, i.\,e. if $d(x+h,x)\leq
d(h,\zero)=\norm{h}$ for all $x,h\in\K$. A stronger assumption is the
translation-invariance of the metric meaning that $d(x+h,y+h)=d(x,y)$
for all $x,y,h\in\K$. In a cone with sub-invariant metric,
$\zero=\neutral$, see \cite[Lemma~2.7]{DMZ}.

In the following, $\S = \{x\in\K: \norm{x}=1\}$ denotes
the unit sphere. For $\eps>0$,
\begin{displaymath}
  A^{\eps} = \{x\in\K: d(x,A)\leq \eps\}
\end{displaymath}
is the $\eps$-envelope of $A\subset\K$, where $d(x,A) = \inf_{a\in
  A}d(x,a)$. The Borel $\sigma$-algebra on $\K$ is denoted by $\sB$
and used to define random cone elements $\xi$ as measurable maps from
a probability space $(\Omega,\sF,\P)$ to $(\K,\sB)$.

Furthermore, $\interior A$, $\closure A$ and $\partial A$ denote the
interior, closure and boundary of $A\subset\K$. A set $A\subset\K$ is
said to be bounded away from a point $x\in\K$ if $x\not\in\closure A$.
If $\mu$ is a measure on $\sB$, then $A\in\sB$ is called a
$\mu$-continuity set if $\mu(\partial A) = 0$.

Let $M_{\zero}$ be the space of all Borel measures on
$\K'=\K\setminus\{\zero\}$ taking finite values on $\K\setminus
\{x\in\K:\norm{x}\geq r\}$ for each $r>0$. By $\sC_0$ we denote the
class of all real-valued bounded continuous functions on $\K'$ with
support bounded away from $\zero$.  A sequence $\{\mu_n,n\geq 1\}$ of
measures from $M_{\zero}$ is said to converge to $\mu\in M_{\zero}$ if
\begin{displaymath}
  \int f\d\mu_n\to \int f\d\mu\quad\text{as } n\to\infty
\end{displaymath}
for all $f\in\sC_0$, equivalently $\mu_n(U)\to\mu(U)$ for all
$\mu$-continuity sets $U\in\sB$ bounded away from $\zero$. 

The following definition does not rely on the semigroup operation and
is available for random elements in metric spaces, where a scaling by
positive real numbers is defined.

\begin{definition}[see \cite{HL}]
  \label{def:regvar}
  A random cone element $\xi$ is called \emph{regularly varying} (at
  infinity) if there exist a non-null measure $\mu\in M_{\zero}$ and a
  sequence $\{a_n, n\geq 1\}$ of positive numbers such that
  \begin{displaymath}
    n\P(\xi \in a_n \,\cdot)\to\mu(\cdot) \quad \text{in } M_{\zero}
    \quad \text{as }n\to\infty.
  \end{displaymath}
\end{definition}

The \emph{tail measure} $\mu$ necessarily scales like a power
function, i.\,e. $\mu(\lambda U) = \lambda^{-\alpha} \mu(U)$ for every
$\lambda>0$, all $\mu$-continuous $U$ bounded away from $\zero$ and
$\alpha>0$ called the \emph{index} of regular variation of
$\xi$.

By \cite[Th.~3.1]{HL}, regular variation of $\xi$ implies
\begin{equation}
  \label{eq:reg2}
  \frac{\P(\xi \in t\, \cdot)}{\P(\norm{\xi} > t)} \to c\mu(\cdot) \text{
    in } M_{\zero} \text{ as } t\to\infty
\end{equation}
for some $c>0$. It will subsequently be assumed that $c=1$ in
(\ref{eq:reg2}), which is possible by scaling
$\{a_n\}$.

By \cite[Th.~3.1]{HL}, $\xi$ is regularly varying with index
$\alpha>0$ if and only if there exist a finite measure $\sigma$
(called the \emph{spectral measure}) on the
unit sphere $\S$ and a sequence $\{\tilde{a}_n\}$ such that
\begin{displaymath}
  \lim_{n\to\infty}n\P(\norm{\xi}^{-1} \xi\in B, \norm{\xi} >
  r \tilde{a}_n) = \sigma(B)\, r^{-\alpha}
\end{displaymath}
for all $r>0$ and all Borel $B\subset\S$ with $\sigma(\partial B)=0$.
It holds that $\tilde{a}_n \sim a_n$.

Karamata's theorem implies the following result. 

\begin{corollary}
  \label{cor:trunc.moment}
  Let $\xi$ be regularly varying with index $\alpha>0$ and let $T>0$
  and $\gamma>\alpha$. Then
  \begin{displaymath}
    \E(\norm{\xi}\I{\norm{\xi}\leq T})^{\gamma} =
    \gamma\int_{0}^{T}\!\!\P(\norm{\xi}>t)\,t^{\gamma-1}\d t \sim
    c\,T^{\gamma}\P(\norm{\xi}>T) \quad \text{as } T\to \infty,
  \end{displaymath}
  where $c>0$ denotes a finite constant.
\end{corollary}

The letter $c$ (also with subscripts) denotes finite, strictly
positive constants; its value may change at every occurrence.

\section{Large deviations with strong scaling}
\label{sec:large-devi-with}

Consider a sequence $\{\xi_n\}_{n\geq 1}$ of \iid random elements in
$\K'$ and their partial sums
\begin{displaymath}
  S_n=\xi_1+\cdots+\xi_n, \qquad n\geq 1.
\end{displaymath}

\begin{theorem}
  \label{thm:1}
  Let $\xi, \xi_1, \ldots$ be \iid regularly varying random elements
  with index $\alpha>0$, tail measure $\mu$ and normalizing sequence
  $\{a_n\}$ in a convex cone with sub-invariant homogeneous metric
  $d$. Let $\{\lambda_n, n\geq 1\}$ be a sequence such that
  \begin{enumerate}
    \renewcommand{\theenumi}{\roman{enumi}}
    \renewcommand{\labelenumi}{(\theenumi)}
  \item $\lambda_n/a_n\to\infty$ if $\alpha < 1$,
  \item $\lambda_n/n\to\infty$, \; $\lambda_n/a_n\to\infty$ and
    $(n/\lambda_n) \E( \norm{\xi}\I{\norm{\xi}\leq \lambda_n})\to 0$
    if $\alpha = 1$,
  \item $\lambda_n/n\to\infty$ if $\alpha > 1$.
  \end{enumerate}
  Then
  \begin{displaymath}
    \gamma_n\P(S_n \in \lambda_n \cdot) \to \mu(\cdot) \quad\text{ in }
    M_{\zero}\quad \text{ as } n\to\infty\,,
  \end{displaymath}
  where $\gamma_n=(n \P(\norm{\xi} > \lambda_n))^{-1}$.
\end{theorem}

The proof closely follows the lines of the proof of
\cite[Th.~1]{MPSa}. Note that the sequence $\{\lambda_n\}$ grows
faster than $n^{\max(1,1/\alpha)}$. 

\begin{lemma}
  \label{lem:sumconv}
  In the setting of Theorem~\ref{thm:1}, $\lambda_n^{-1}\norm{S_n} \to
  0$ in probability.
\end{lemma}
\begin{proof}
  The sub-invariance of the metric implies the sub-linearity of the
  norm, thus it suffices to show that
  \begin{equation}
    \label{eq:subinvar}
    \frac{\norm{\xi_1}+\cdots+\norm{\xi_n}}{\lambda_n}
    \to 0 \text{ in probability}.
  \end{equation}
  If $\alpha>1$, then $\E\norm{\xi}<\infty$ and the strong law of
  large numbers with the growth conditions on $\{\lambda_n\}$ provides
  the result.

  Assume $0<\alpha\leq 1$. By \cite[Th.~4.13]{PVV},
  (\ref{eq:subinvar}) holds if (and only if) the following three
  conditions hold.
  \begin{description}
  \item[\rm{(i)}] $n\P(\norm{\xi}>\lambda_n)\to 0$. This is the case, since 
    \begin{displaymath}
      n\P(\norm{\xi}>\lambda_n) = n\P(\norm{\xi}>a_n)\;
      \frac{\P(\norm{\xi}>\lambda_n)}{\P(\norm{\xi}>a_n)},
    \end{displaymath}
    where the first factor converges to one and the fraction converges
    to zero.
  \item[\rm{(ii)}] $\lambda_n^{-1}
    n\E(\norm{\xi}\I{\norm{\xi}<\lambda_n})\to0$, which follows from
    Corollary~\ref{cor:trunc.moment} in case $\alpha<1$, while for
    $\alpha=1$ the convergence is assumed.
  \item[\rm{(iii)}] $\lambda_n^{-2} n\Var(\norm{\xi}
    \I{\norm{\xi}<\lambda_n})\to0$. To confirm this, bound the
    variance by the second moment and apply
    Corollary~\ref{cor:trunc.moment}.\qedhere
  \end{description}
\end{proof}

\begin{proof}[Proof of Theorem~\ref{thm:1}]
  Let $U\in\sB$ with $U\neq\emptyset$, $\mu(\partial U) = 0$ and $\zero \not\in \closure
  U$. We start by bounding $\gamma_n\P(S_n\in\lambda_n U)$ from
  above. For any $\eps>0$,
  \begin{align*}
    \P(S_n\in\lambda_n U)&=\P(S_n\in\lambda_n
    U,\mathop{\cup}_{i=1}^n\{\xi_i\in\lambda_n U^{\eps}\}) + \P(S_n\in\lambda_n
    U,\mathop{\cap}_{i=1}^n\{\xi_i\not\in\lambda_n U^{\eps}\})\\
    &\leq n\P(\xi_1\in\lambda_n U^{\eps}) + \P(\mathop{\cap}_{i=1}^n
    \{d(S_n,\xi_i)>\eps\lambda_n\})\\
    &=I_1+I_2,
  \end{align*}
  since $S_n\in\lambda_nU$ and $\xi_i\not\in\lambda_nU^{\eps}$ imply
  that $d(\lambda_n^{-1}S_n,\lambda_n^{-1}\xi_i)>\eps$ so that
  $d(S_n,\xi_i)>\eps \lambda_n$ by the homogeneity of $d$.

  By (\ref{eq:reg2}), $\gamma_nI_1\to \mu(U^{\eps})$ as $n\to\infty$
  for $\mu$-continuity sets $U^{\eps}$ with $\zero
  \not\in\closure U^{\eps} $. Note that $\zero \not\in\closure U^{\eps}$ for
  sufficiently small $\eps$ and that all but countably many
  $U^{\eps}$-sets are $\mu$-continuity sets since $\mu$ is finite
  outside any neighborhood of $\zero$. It follows that $\mu(U^{\eps})
  \to \mu(U)$ as $\eps \downarrow 0$.

  To show that $\gamma_nI_2\to 0$ as $n\to\infty$ for every $\eps>0$,
  consider for $\delta>0$ the following events partitioning the
  probability space:
  \begin{align*}
    D_1 &=\bigcup_{1\leq i<j\leq n}\{\norm{\xi_i}>\delta\lambda_n,
    \norm{\xi_j}>\delta\lambda_n\},\\
    D_2 &=\bigcup_{i=1}^n\left\{\norm{\xi_i}>\delta\lambda_n,
      \norm{\xi_j}\leq\delta\lambda_n,j\neq
      i,j=1,\ldots,n\right\},\\
    D_3 &=\left\{\max_{i=1,\ldots,n}\norm{\xi_i}\leq\delta\lambda_n\right\}.
  \end{align*}
  By the Bonferroni inequality and the independence of the $\xi_i$,
  \begin{displaymath}
    \gamma_n\P(D_1)\leq\binom{n}{2}
    \frac{\P(\norm{\xi_1}>\delta\lambda_n)^2}{n\P(\norm{\xi}>\lambda_n)},
  \end{displaymath}
  which converges to zero as $n\to\infty$ because of the regular
  variation property and the growth condition on $\{\lambda_n\}$.

  By the sub-invariance of the metric, $d(S_n,\xi_n)\leq
  \norm{S_{n-1}}$. Therefore,
  \begin{align*}
    \P(\mathop{\cap}_{i=1}^n \{d(S_n,\xi_i)>\eps\lambda_n\}, D_2) &\leq
    \sum_{i=1}^n
    \P(d(S_n,\xi_i)>\eps\lambda_n,\norm{\xi_i}>\delta\lambda_n)\\
    &\leq n\P(\norm{S_{n-1}}>\eps\lambda_n)\P(\norm{\xi_n}>\delta\lambda_n)
  \end{align*}
  which, if multiplied by $\gamma_n$, converges to zero as
  $n\to\infty$ because of the regular variation of $\xi$ and
  Lemma~\ref{lem:sumconv}.

  Regarding $D_3$, by sub-invariance
  \begin{displaymath}
    \P\left(\mathop{\cap}_{i=1}^n \{d(S_n,\xi_i)>\eps\lambda_n\}, 
      \max_{i=1,\ldots,n}\norm{\xi_i}\leq\delta\lambda_n\right) \leq\P\left(\sum_{i=1}^{n}\norm{\xi_i}\I{\norm{\xi_i}\leq\delta\lambda_n}
      >\eps\lambda_n\right),
  \end{displaymath}
  which, after centering, becomes 
  \begin{displaymath}
    \P\left(\sum_{i=1}^{n}\left(\norm{\xi_i}\I{\norm{\xi_i}\leq\delta\lambda_n}-
      \E(\norm{\xi_i}\I{\norm{\xi_i}\leq\delta\lambda_n})\right)>
      \eps\lambda_n-n\E(\norm{\xi_1}\I{\norm{\xi_1}\leq\delta\lambda_n})\right).
  \end{displaymath}
  Since
  $n\lambda_n^{-1}\E(\norm{\xi_1}\I{\norm{\xi_1}\leq\delta\lambda_n})\to
  0$ (see proof of Lemma~\ref{lem:sumconv}), the right hand side of
  the above inequality may be replaced by $\eps\lambda_n/2$.

  It remains to show that
  \begin{displaymath}
    \gamma_n\P\left(\sum_{i=1}^n\left(\norm{\xi_i}\I{\norm{\xi_i}
        \leq\delta\lambda_n}-\E(\norm{\xi_i}
      \I{\norm{\xi_i}\leq\delta\lambda_n})\right)>\eps\lambda_n/2\right)\to
    0 \text{ as } n\to\infty.
  \end{displaymath}
  Since each summand
  \begin{displaymath}
    \eta_i=\norm{\xi_i}\I{\norm{\xi_i}\leq\delta\lambda_n}
    -\E(\norm{\xi_i}\I{\norm{\xi_i}\leq\delta\lambda_n})
  \end{displaymath}
  is centered and $\E|\eta_i|^p<\infty$ for any $p\geq 2$, the
  Fuk--Nagaev inequality (see e.\,g. \cite[p.~78]{PVV}) yields that
  \begin{displaymath}
    \P\left(\sum_{i=1}^n \eta_i>\frac{\eps\lambda_n}{2}\right) \leq
    c_1 n(\eps\lambda_n)^{-p}\, \E|\eta_1|^p
    +\exp\left\{-\frac{c_2 (\eps\lambda_n)^2}{(n \Var \eta_1)}\right\}
    =I_{3,1}+I_{3,2}
  \end{displaymath}
  for any $p\geq 2$ where $c_1,c_2>0$ are finite constants.
  
  By Corollary~\ref{cor:trunc.moment},
  \begin{displaymath}
    \E |\eta_1|^p\leq
    \E(\norm{\xi_1}\I{\norm{\xi_1}\leq\delta\lambda_n})^p \sim
    c(\delta\lambda_n)^p\P(\norm{\xi_1}>\delta\lambda_n)
    \quad \text{as } n\to\infty
  \end{displaymath}
  for $p>\alpha$.
  For $p>\max\{2,\alpha\}$,
  \begin{displaymath}
    \lim_{\delta\downarrow 0}\limsup_{n\to\infty} \gamma_nI_{3,1}
    \leq c\lim_{\delta\downarrow 0}\delta^p\lim_{n\to\infty}
    \frac{\P(\norm{\xi_1}>\delta\lambda_n)}{\P(\norm{\xi_1}>\lambda_n)}
    =c\lim_{\delta\downarrow 0} \delta^{p-\alpha}=0.
  \end{displaymath}
  To show that $\limsup_{n\to\infty}\gamma_nI_{3,2}= 0$, consider
  these (disjoint) cases:
  \begin{enumerate}
    \renewcommand{\theenumi}{\roman{enumi}}
    \renewcommand{\labelenumi}{(\theenumi)}
  \item If $\alpha\geq2$ and $\Var \norm{\xi_1} < \infty$, then $\Var
    \eta_1 < \infty$ and the convergence follows.
  \item If $0<\alpha<2$, then $\lambda_n^{-2}n
    \Var(\norm{\xi_1}\I{\norm{\xi_1}\leq\delta\lambda_n}) \sim
    cn\P(\norm{\xi}>\delta\lambda_n)$ by
    Corollary~\ref{cor:trunc.moment}, which implies the
    convergence.
  \item If $\alpha=2$ and $\Var \norm{\xi}=\infty$, then
    $\lambda_n^2\P(\norm{\xi}>\lambda_n)$ and
    $\Var(\norm{\xi}\I{\norm{\xi}\leq\delta\lambda_n})$ are both
    slowly varying functions of $\lambda_n$. Because
    $\lambda_n/n\to\infty$, the convergence follows.
  \end{enumerate}
  Thus $\limsup_{n\to\infty}\gamma_nI_{3,2}= 0$. Hence
  \begin{displaymath}
    \limsup_{n\to\infty} \gamma_n\P(S_n\in\lambda_nU) \leq
    \mu(U^{\eps})\to\mu(U)\quad\text{ as } \eps\downarrow 0
  \end{displaymath}
  for any $U$ bounded away from $\zero$, establishing the upper bound.

  For the lower bound, let $U\in\sB$ now denote a $\mu$-continuity set
  bounded away from $\zero$ with nonempty interior. The set
  $U^{-\eps}=((U^c)^{\eps})^c$ is bounded away from
  $\zero$, is a $\mu$-continuity set for all but countably many $\eps$ and
  $\interior(U^{-\eps})$ is nonempty for sufficiently small $\eps$.

  Writing $S_n^{\neq i}=\!\!\sum\limits_{j=1,j\neq i}^n \!\!\!\xi_j$ for
  $i=1,\ldots,n$, we see that
  \begin{align*}
    \P(S_n \in \lambda_n U)&\geq\P(S_n \in \lambda_n U,
    \,\cup_{i=1}^n\{\xi_i\in\lambda_n U^{-\eps}\})\\
    &\geq \P(\cup_{i=1}^n\{d(S_n,\xi_i)<\eps\lambda_n,\xi_i\in\lambda_nU^{-\eps}\})\\
    &\geq \P(\cup_{i=1}^n\{\norm{S_n^{\neq
        i}}<\eps\lambda_n,\xi_i\in\lambda_nU^{-\eps}\})\\
    &\geq n\P(\norm{S_n^{\neq
        1}}<\eps\lambda_n)\P(\xi_1\in\lambda_nU^{-\eps})- {n \choose 2}
    \P(\xi_1\in\lambda_nU^{-\eps})^2\,,\\
    &=J_1-J_2,
  \end{align*}
  where the second inequality holds because
  $d(S_n,\xi_i)<\eps\lambda_n$ and $\xi_i\in\lambda_n U^{-\eps}$ imply that
  $\lambda_n^{-1}S_n\in (U^{-\eps})^{\eps}\subset U$. The third
  inequality is implied by the sub-invariance of the metric.

  By Lemma~\ref{lem:sumconv}, $\P(\norm{S_n^{\neq 1}}<\eps\lambda_n)\to
  1$ as $n\to\infty$ and
  \begin{displaymath}
    \liminf_{n\to\infty}\gamma_nJ_1
    \geq \lim_{n\to\infty}\frac{\P(\xi_1\in\lambda_nU^{-\eps})}
    {\P(\norm{\xi}>\lambda_n)}=\mu(U^{-\eps})
  \end{displaymath}
  which converges to $\mu(U)$ as $\eps\downarrow 0$.
  Finally,
  \begin{displaymath}
    \limsup_{n\to\infty} \gamma_nJ_2\leq c \lim_{n\to\infty}
    n\P(\xi_1\in\lambda_n U^{-\eps}) = 0,
  \end{displaymath}
  which establishes the lower bound and finishes the proof.
\end{proof}

\section{Moderate scaling}
\label{sec:moderate-scaling}

Theorem~\ref{thm:1} requires that the normalising sequence
$\{\lambda_n\}$ grows faster than $n$ in case $\alpha\geq 1$.  If
$\lambda_n$ grows slower than $n$, but faster than
$n^{\max(1/\alpha,1/2)}$, then the large deviation result holds with
an extra additive normalization. Care is required however, since
the addition operation in general cones is not invertible and the
expectation is not well defined.

\begin{theorem}
  \label{thm:2}
  Let $\xi, \xi_1, \ldots$ be \iid regularly varying random elements
  with index $\alpha\geq 1$ and tail measure $\mu$ in a convex cone
  $\K'$ with a homogeneous sub-invariant metric $d$.  Assume that
  $\E\norm{\xi}$ is finite and there exists a sequence $\{A_n, n\geq
  1\}$ of cone elements such that
  \begin{equation}
    \label{cond:conv2} 
    \lambda_n^{-1}\E d(S_n,A_n) \to 0\qquad \text{as } n\to\infty\,,
  \end{equation}
  where $\lambda_n/n^{\max \{1/\alpha,1/2\} + \eta} \to\infty$ for
  some $\eta>0$.  If
  \begin{center}
    (A) $A_n=\neutral$ for all $n$, \quad or \quad (B) the metric $d$ is
    invariant
  \end{center}
  then, with $\gamma_n=(n \P(\norm{\xi} > \lambda_n))^{-1}$,
  \begin{displaymath}
    \gamma_n\P(S_n \in \lambda_n \cdot + A_n) \to \mu(\cdot) \quad\text{ in }
    M_{\zero}\quad \text{as } n\to\infty.
  \end{displaymath}
\end{theorem}

The following result known in the setting of Banach spaces (see
\cite[Lemma~6.16]{LT}) extends to general semigroups and will be used
to prove Theorem~\ref{thm:2}.

Let $\zeta_1,\dots,\zeta_n$ be integrable random elements in a
semigroup $\K$ with sub-invariant metric $d$. For $i\leq n$ write
$\sA_i=\sigma(\zeta_1,\dots,\zeta_i)$ and let $\sA_0$ denote the
trivial sigma-algebra. Write $S_n^{\neq i}=\sum_{j=1,j\neq i}^n
\zeta_j$ for $i\leq n$. Let $z$ be any fixed cone element and define
\begin{displaymath}
  d_i=\cE{d(S_n,z)}{\sA_i} - \cE{d(S_n,z)}{\sA_{i-1}}\,,
  \qquad i=1,\dots,n.
\end{displaymath}
Then $d_1,\dots,d_n$ is a real-valued martingale difference sequence
and $\sum_{i=1}^n d_i = d(S_n,z)-\E d(S_n,z)$ almost surely.

\begin{lemma}
  \label{lem:lt}
  Let $\zeta_1,\dots,\zeta_n$ be independent. Then
  \begin{displaymath}
    |d_i| \leq \norm{\zeta_i} + \E\norm{\zeta_i}
  \end{displaymath}
  almost surely for every $i=1,\dots,n$.
\end{lemma}
\begin{proof}
  Since $S_n^{\neq i}$ is independent of $\zeta_i$,
  \begin{align*}
    d_i&=\cE{d(S_n,z)}{\sA_i} - \cE{d(S_n,z)}{\sA_{i-1}}\\
    & \quad + \cE{d(S_n^{\neq i},z)}{\sA_{i-1}} - \cE{d(S_n^{\neq i},z)}{\sA_{i}}\\
    &=\cE{d(S_n,z)-d(S_n^{\neq i},z)}{\sA_{i}} -
    \cE{d(S_n,z)-d(S_n^{\neq i},z)}{\sA_{i-1}},
  \end{align*}
  where the equalities hold almost surely. The sub-invariance property
  yields that $d(S_n,z)-d(S_n^{\neq i},z) \leq d(S_n,S_n^{\neq i})\leq
  \norm{\zeta_i}$, giving the result.
\end{proof}

\begin{proof}[Proof of Theorem~\ref{thm:2}]
  We start with an upper bound. Let $\eps>0$ and $U\in\sB$ be a nonempty
  $\mu$-continuity set bounded away from $\zero$. Then
  \begin{align*}
    \begin{split}
      \P(S_n\in\lambda_nU+A_n)&=\P(S_n\in\lambda_nU+A_n,\,
      \cup_{i=1}^n\{\xi_i\in\lambda_n U^{\eps}\})\\
      &\quad+\P(S_n\in\lambda_nU+A_n,\,
      \cap_{i=1}^n\{\xi_i\not\in\lambda_nU^{\eps}\})
    \end{split}\\
    &\leq n\P(\xi_1\in\lambda_nU^{\eps}) + \P(S_n\in\lambda_nU+A_n,\,
    \cap_{i=1}^n\{\xi_i\not\in\lambda_nU^{\eps}\})\\
    &=I_1+I_2.
  \end{align*}
  As in the proof of Theorem~\ref{thm:1}, using (\ref{eq:reg2}) and
  the $\mu$-continuity of $U$,
  \begin{displaymath}
    \mu(U)\leq \lim_{\eps\downarrow 0}\liminf_{n\to\infty}
    \gamma_nI_1\leq \lim_{\eps\downarrow 0}\limsup_{n\to\infty}
    \gamma_nI_1=\mu(U).
  \end{displaymath}
  Now fix $0<\delta\leq \eps/3$ and partition $\Omega$ into
  \begin{displaymath}
    D_1 = \bigcup_{i=1}^n\left\{\norm{\xi_i}>\delta\lambda_n\right\}
    \text{ and }
    D_2 = \bigl\{\max_{i=1,\ldots,n} \norm{\xi_i} \leq \delta\lambda_n\bigr\}.
  \end{displaymath}
  Starting with $D_1$, we see that 
  \begin{align*}
    \P(S_n\in\lambda_nU+A_n,&\, \cap_{k=1}^n\{\xi_k
    \not\in\lambda_nU^{\eps}\},\, D_1)\\ &\leq n\P(S_n\in\lambda_nU+A_n,\,
    \xi_1\not\in\lambda_nU^{\eps},\, \norm{\xi_1}>\delta\lambda_n)\\ &\leq
    n\P(d(\xi_2+\cdots+\xi_n,A_n)>\eps\lambda_n,\, \norm{\xi_1}>\delta\lambda_n).
  \end{align*}
  To justify the last step, define the event
  \begin{displaymath}
    C=\{S_n\in\lambda_nU + A_n,\,\xi_1\not\in\lambda_nU^{\eps}\}.
  \end{displaymath}
  If (A) holds,
  $C=\{\lambda_n^{-1}S_n\in
  U,\,\lambda_n^{-1}\xi_1\not\in U^{\eps}\}$ so that
  $d(S_n,\xi_1)>\eps\lambda_n$ by the homogeneity of $d$. Now
  $d(S_n,\xi_1)\leq d(\xi_2+\cdots+\xi_n,A_n)$ by the sub-invariance.
  
  If (B) holds, $\xi_1\not\in\lambda_nU^{\eps}$ if and only if $\xi_1
  + A_n \not\in\lambda_nU^{\eps} + A_n$ so that
  \begin{displaymath}
    C=\{S_n\in\lambda_nU+A_n,\,\xi_1+A_n\not\in\lambda_nU^{\eps}+A_n\},
  \end{displaymath}
  implying that $d(S_n,\xi_1+A_n)>\eps\lambda_n$. Now apply the
  invariance of $d$ again.

  By independence of the $\xi_i$,
  \begin{align*}
    \gamma_nn\P(d(\xi_2+&\cdots+\xi_n,A_n)>\eps\lambda_n,\,
    \norm{\xi_1}>\delta\lambda_n)\\
    &=\P(\lambda_n^{-1}d(\xi_2+\cdots+\xi_n,A_n)>\eps)\;
    \frac{\P(\norm{\xi_1}>\delta\lambda_n)}{\P(\norm{\xi}>\lambda_n)}.
  \end{align*}
  The fraction in the right-hand side converges to $\delta^{-\alpha}$
  as $n\to\infty$.  The first factor converges to zero, since
  \begin{displaymath}
    \lambda_n^{-1}d(S_n^{\neq 1},A_n)\leq
    \lambda_n^{-1}\norm{\xi_1} + \lambda_n^{-1}d(S_n,A_n)
  \end{displaymath}
  converges to zero in probability because
  $\lambda_n^{-1}\norm{\xi_1}\to 0$ in probability and
  (\ref{cond:conv2}) implies that $\lambda_n^{-1}d(S_n,A_n)\to 0$
  in probability.

  By the same reasoning as for $D_1$, for $D_2$ we get
  \begin{displaymath}
    \P(S_n\in\lambda_nU+A_n,\,
    \cap_{k=1}^n\{\xi_k\not\in\lambda_nU^{\eps}\},\, D_2)\leq
    \P(d(\xi_2+\cdots+\xi_n,A_n)>\eps\lambda_n,D_2).
  \end{displaymath}
  Since $d(S_n^{\neq 1},A_n) \leq d(S_n^{\neq 1},S_n) + d(S_n,A_n)$
  and $d(S_n^{\neq 1},S_n)\leq \norm{\xi_1}$ the sub-invariance
  property yields that
  \begin{align*}
    \P(d(S_n^{\neq 1},A_n)>\eps\lambda_n,\,\max_{i=1,\ldots,n}
    &\norm{\xi_i} \leq \delta\lambda_n)\\ 
    &\leq\P(d(S_n,A_n)>\eps\lambda_n/2,\,\max_{i=1,\ldots,n} \norm{\xi_i}
    \leq \delta\lambda_n)\\
    &\leq\P(d(S_n,A_n)>\eps\lambda_n/2,\,\max_{i=1,\ldots,n} d(\xi_i,A_1)
    \leq 2\delta\lambda_n)
  \end{align*}
  for sufficiently large $n$. The latter inequality holds since
  $d(\xi_i,A_1) \leq \norm{\xi_i} + \norm{A_1}$ and $\delta\lambda_n$
  eventually exceeds $\norm{A_1}$.  Defining
  \begin{displaymath}
    \xi_i^{\delta} =
    \begin{cases}
      \xi_i,&d(\xi_i,A_1)\leq 2\delta\lambda_n,\\
      \neutral,&d(\xi_i,A_1)>2\delta\lambda_n,
    \end{cases}
    \quad i=1,\dots,n\,
  \end{displaymath}
  and $S_n^{\delta} = \xi_1^{\delta} + \cdots + \xi_n^{\delta}$, we
  see that it suffices to show that, for $\delta$ sufficiently small,
  $\gamma_n\P(d(S_n^{\delta}, A_n)>\eps\lambda_n)\to 0$ as
  $n\to\infty$. For this, show that $\E
  d(S_n^{\delta},A_n)\leq\eps\lambda_n/2$ for sufficiently large $n$.

  By the triangle inequality,
  \begin{displaymath}
    \lambda_n^{-1}\E d(S_n^{\delta},A_n)\leq
    \lambda_n^{-1}\E d(S_n^{\delta},S_n)+\lambda_n^{-1}\E d(S_n,A_n),
  \end{displaymath}
  where the last term converges to zero as $n\to\infty$ by
  \eqref{cond:conv2}. To show that also $\lambda_n^{-1}\E
  d(S_n^{\delta},S_n)\to 0$ as $n\to\infty$ for fixed $\delta$, let
  $I(n)=\{1\leq i\leq n: d(\xi_i,A_1)>2\delta\lambda_n\}$, so that
  \begin{displaymath}
    d(S_n^{\delta},S_n) \leq \biggl\|\sum_{i\in I(n)} \xi_i\biggr\|
  \end{displaymath}
  by sub-invariance. 

  If $\alpha=1$, $\lambda_n^{-1}\E\norm{\sum_{i\in I(n)} \xi_i}
  \leq\lambda_n^{-1}n\E\norm{\xi_1}$, which converges to zero.  In
  case $\alpha>1$, we have to show that
  \begin{displaymath}
    \lambda_n^{-1}n\E(\norm{\xi_1} \I{d(\xi_1,A_1)>2\delta\lambda_n})\to
    0\quad \text{as } n\to\infty.
  \end{displaymath}
  Since $d(\xi_1,A_1)\geq \norm{\xi_1}-\norm{A_1}$, it suffices that
  $\lambda_n^{-1}n\E(\norm{\xi_1} \I{\norm{\xi_1}>3\delta\lambda_n})$
  converges to 0, which follows from Corollary~\ref{cor:trunc.moment}
  and the growth rate of $\{\lambda_n\}$.

  Now we assume that $n$ is so large that $\E
  d(S_n^{\delta},A_n)\leq\eps\lambda_n/2$. Then
  \begin{displaymath}
    \P(d(S_n^{\delta}, A_n)>\eps\lambda_n)
    \leq \P(d(S_n^{\delta},A_n)-\E d(S_n^{\delta},A_n) >\eps\lambda_n/2).
  \end{displaymath}
  Note that $d(S_n^{\delta},A_n)-\E d(S_n^{\delta},A_n)$ is almost
  surely equal to $\sum_{i=1}^n d_i$ with
  $d_i=\cE{d(S_n^{\delta},A_n)}{\sA_i} -
  \cE{d(S_n^{\delta},A_n)}{\sA_{i-1}}$. Lemma~\ref{lem:lt} (applied
  for $\zeta_i=\xi_i^{\delta}$ and $z=A_n$) yields that $|d_i|\leq
  \norm{\xi_i^{\delta}}+\E\norm{\xi_i^{\delta}}$, which is almost
  surely smaller than $6\delta\lambda_n$ for sufficiently large $n$.

  A martingale version of Bennett's inequality \cite[Eq.~(6.13)]{LT}
  for the martingale difference sequence $d_1,\dots,d_n$ yields that 
  \begin{align*}
    \P(d(S_n^{\delta},A_n&)-\E d(S_n^{\delta},A_n) >\eps\lambda_n/2)\\
    &\leq 2\exp\left\{\frac{\eps}{12\delta}-\left(\frac{\eps}{12\delta}
        +\frac{b}{144\delta^2\lambda_n^2}\right)
      \log\left(1+\frac{6\delta\eps\lambda_n^2}{b}\right)\right\}
  \end{align*}
  for any $b\geq 4n\E\norm{\xi_1^{\delta}}^2 \geq \sum_{i=1}^n \E d_i^2$.

  In case $\alpha\geq 2$ and $\E\norm{\xi}^2<\infty$, we choose
  $b=4n\E\norm{\xi}^2$. In case $1\leq\alpha\leq 2$ and
  $\E\norm{\xi}^2$ is infinite, we choose $b=8n(\norm{A_1}^2 +
  \E\norm{\xi_1^{\delta}}^2\I{d(\xi_1,A_1)\leq
    2\delta\lambda_n})$. Using the regular variation of $\norm{\xi}$,
  we see that in both cases the growth conditions on $\{\lambda_n\}$
  provide that for all $\delta>0$ small enough,
  \begin{displaymath}
    \gamma_n\P(d(S_n^{\delta},A_n)-\E d(S_n^{\delta},A_n) >\eps\lambda_n/2) \to
    0\quad \text{as }n\to\infty.
  \end{displaymath}
  This establishes the upper bound. 

  For the lower bound, we let $U\in\sB$ again be a $\mu$-continuity set
  bounded away from $\zero$ with nonempty interior and write
  \begin{align*}
    \P(S_n&\in\lambda_nU + A_n) \geq
    \P(S_n\in\lambda_nU+A_n,\cup_{i=1}^n\{\xi_i\in\lambda_nU^{-\eps}\})\\
    &=\P(\cup_{i=1}^n\{\xi_i\in\lambda_nU^{-\eps}\}) -
    \P(S_n\not\in\lambda_nU +
    A_n,\,\cup_{i=1}^n\{\xi_i\in\lambda_nU^{-\eps}\})\\
    &=I_1-I_2.
  \end{align*}
  To show that $\lim_{\eps\downarrow 0}\limsup_{n\to\infty}
  \gamma_nI_2 = 0$, first note that
  \begin{displaymath}
    \P(S_n\not\in\lambda_nU +
    A_n,\,\cup_{i=1}^n\{\xi_i\in\lambda_nU^{-\eps}\}) \leq
    n\P(S_n\not\in\lambda_nU + A_n,\xi_1\in\lambda_nU^{-\eps}).
  \end{displaymath}
  Now $S_n\not\in\lambda_nU +A_n,\xi_1\in\lambda_nU^{-\eps}$ implies
  that $d(S_n^{\neq 1},A_n)>\eps\lambda_n$, by separately considering
  cases (A) and (B) as above. Then the same arguments as in the upper
  bound part of the proof apply.

  It remains to show that
  $\limsup_{n\to\infty}\gamma_nI_1\geq\mu(U^{-\eps})$. By a Bonferroni
  argument,
  \begin{displaymath}
    \gamma_nI_1 \geq
    \frac{\P(\xi_1\in\lambda_nU^{-\eps})}{\P(\norm{\xi_1}>\lambda_n)} -
    \frac{n-1}{2}\frac{\P(\xi_1\in\lambda_nU^{-\eps})^2}
    {\P(\norm{\xi_1}>\lambda_n)}.
  \end{displaymath}
  By choosing $\eps$ sufficiently small (such that $U^{-\eps}$ is a
  $\mu$-continuity set bounded away from $\zero$), we may apply
  \eqref{eq:reg2} and conclude that the $\limsup$ of the positive
  summand converges to $\mu(U^{-\eps})$. The upper limit of the second
  term is zero because $\lambda_n/a_n\to\infty$. Letting $\eps\to 0$
  establishes the lower bound.
\end{proof}

A convex cone with metric $d$ is said to be \emph{isometrically
  embeddable} in a Banach space $(\B,\norm{\cdot}_{\B})$ if there
exists a measurable map $I:\K\to\B$ such that $I(x+y) = I(x) + I(y)$
and $d(x,y) = \norm{I(x)-I(y)}_{\B}$ for all $x,y\in\K$. If the second
distributivity law in $\K$ holds, this is always possible and then $I$
becomes a linear map, i.\,e. $I(ax)=aI(x)$ for all $a>0$.

By \cite[Prop.~9.11]{LT}, a Banach space $\B$ of (Rademacher) type
$p\in[1,2]$ has the property that for every finite sequence
$X_1,\ldots,X_n$ of independent mean zero $p$-integrable Radon random
variables in $\B$,
\begin{equation}
  \label{eq:type}
  \E\norm{\sum_{i=1}^n X_i}^p_{\B}
  \leq (2c)^p \sum_{i=1}^n \E\norm{X_i}^p_{\B}.
\end{equation}
If $\xi$ is a random element in $\K$ with $\E\norm{\xi}<\infty$, then
$I(\xi)$ is a random element in $\B$ which is strongly integrable with
expectation $\E I(\xi)$.

\begin{corollary}
  \label{cor:banach}
  Let $\xi,\xi_1,\ldots$ be \iid regularly varying random elements
  with index $\alpha\geq 1$ and tail measure $\mu$ in a convex cone
  $\K$ which is isometrically embeddable by $I:\K\to\B$ in a separable
  Banach space $\B$ of type $\min(\alpha,2)$. Assume that
  $\E\norm{\xi}^{\min(\alpha,2)}$ is finite and $I(A_n)=n\E I(\xi)$,
  $n\geq1$, for some cone elements $\{A_n,n\geq1\}$. Then
  \eqref{cond:conv2} holds, and so the statement of
  Theorem~\ref{thm:2} follows.
\end{corollary}
\begin{proof}
  Note that the existence of the embedding implies that the metric $d$
  is invariant.  If $\tilde{S}_n = I(S_n) - n\E I(\xi)$, then
  \begin{displaymath}
    \E d(S_n,A_n) =\E\norm{\tilde{S}_n}_{\B}.
  \end{displaymath}
  Since $\B$ is a Banach space of type $\min(\alpha,2)$ and each
  summand $X_i=I(\xi_i) - \E I(\xi)$, $i=1,\dots,n$, is centered and Radon
  by the separability of $\B$, it holds that
  \begin{displaymath}
    \E\norm{\tilde{S}_n}_{\B}^{\min(\alpha,2)}
    \leq c n \E\norm{X_1}_{\B}^{\min(\alpha,2)}
  \end{displaymath}
  for some finite $c$ by \eqref{eq:type}. Further
  \begin{displaymath}
    \sup_n n^{-1}\E\norm{\tilde{S}_n}_{\B}^{\min(\alpha,2)} <\infty.
  \end{displaymath}
  By Jensen's inequality, 
  \begin{displaymath}
    \sup_n n^{-1/\min(\alpha,2)+\eta}\E\norm{\tilde{S}_n}_{\B} \to 0
  \end{displaymath}
  for any $\eta>0$, so that \eqref{cond:conv2} holds.
\end{proof}

\section{Examples}\label{sec:examples}

\begin{example}[Half-line with maximum]
  \label{ex:max}
  Let $\K=[0,\infty)$ with the semigroup operation $x+y=\max\{x,y\}$
  and the usual multiplication, so that
  $S_n=\max\{\xi_1,\dots,\xi_n\}$. The metric $d(x,y)=|x-y|$ is
  homogeneous and sub-invariant, and $\neutral=\zero=0$. Regularly
  varying random elements are precisely the nonnegative random
  variables with regularly varying right tail of index $\alpha>0$, so
  their distributions are in the maximum domain of attraction of the
  Fr\'{e}chet distribution and $a_n=n^{1/\alpha}\ell(n)$ for some
  slowly varying function $\ell$. Thus, Theorem~\ref{thm:1} applies.

  Since the metric is not invariant, condition (A) in
  Theorem~\ref{thm:2} is imposed, i.\,e. $A_n=0$ for all $n$.

  Condition \eqref{cond:conv2} requires that $\lambda_n^{-1}\E
  \max\{\xi_1,\dots,\xi_n\}\to 0$. If $\alpha=1$, it follows from the
  law of large numbers, since $\max\{\xi_1,\dots,\xi_n\}/\lambda_n\leq
  (\xi_1+\cdots+\xi_n)/n^{1+\eta}$ and $\E\xi_1<\infty$.  In case
  $\alpha>1$, we can use the fact that $\xi_1^{\alpha-\eps}$ is
  integrable for any $\eps>0$. Since $\alpha-\eps\geq1$ for
  sufficiently small $\eps>0$, Jensen's inequality yields that 
  \begin{displaymath}
    n^{-1-\eta'}(\E\max(\xi_1,\dots,\xi_n))^{\alpha-\eps}\to 0
  \end{displaymath}
  for each $\eta'>0$ and it remains to choose $\eta'$ and $\eps>0$
  such that $(1+\eta')/(\alpha-\eps)\leq 1/\alpha+\eta$.
\end{example}

\begin{example}[Compact sets in $\R^m$]
  The cone of compact sets in the Euclidean space with the Minkowski
  addition metrized by the Hausdorff metric falls into the scheme of
  Theorem~\ref{thm:1}. This case is considered in \cite{MPSa}.
\end{example}

\begin{example}[Compact convex sets in $\R^m$]
  \label{ex:ccs}
  Let $\K$ consist of nonempty compact convex sets in $\R^m$, equipped
  with Minkowski addition and the usual scaling. The support function
  of the set $X$ is denoted by $h_X(u)$ for $u$ from the unit sphere
  $\S^{m-1}$ in $\R^m$. The Hausdorff distance between compact convex
  sets equals the uniform distance between their support functions, so
  it is possible to embed $\K$ into the Banach space of continuous
  functions on the unit sphere. This argument has been used in
  \cite{MPSb} to derive large deviation results for random convex
  compact sets with integrable norm, which is also a special case of
  Theorem~\ref{thm:2}.

  However, it is possible to get rid of condition \eqref{cond:conv2}
  by considering another metric for convex compact sets. For $p\in
  [1,\infty)$, define the distance between convex compact sets $X$ and
  $Y$ using the $L_p$-distance between their support functions as
  \begin{displaymath}
    d_p(X,Y) = \left(\int_{\S^{m-1}} |h_X(u)-h_Y(u)|^p\d u\right)^{1/p},
  \end{displaymath}
  see also \cite{vi85}.  Note that $d_p$ is homogeneous and invariant
  and the support function provides an isometric embedding of $\K$
  into the space $\B=L_p(\S^{m-1})$, Since $\B$ is a separable Banach
  space of type $\min(p,2)$, Corollary~\ref{cor:banach} applies and
  condition \eqref{cond:conv2} is not needed. 
\end{example}

\begin{example}[Functions with argument rescaling]\label{ex:functions}
  Let $\K$ consist of continuous functions $f:\R_+\to\R$ such that
  $\int_{0}^{\infty}xf(x)^2\d x<\infty$, i.\,e. $f\in L_2(R_+,\mu)$, where
  $\mu(\d x) = x \d x$.  The addition is defined
  pointwisely (so $\neutral$ is the zero function) and the
  cone multiplication $\cdot$ is defined as 
  \begin{displaymath}
    (a\cdot f)(x)=f\bigl(\frac{x}{a}\bigr)
  \end{displaymath}
  for $a>0$. The metric
  \begin{displaymath}
    d(f,g)=\left(\int_{0}^{\infty} x(f(x)-g(x))^2 \d x\right)^{1/2}
  \end{displaymath}
  is invariant and homogeneous, and $\K$ is isometrically embeddable
  into the space $\B=L_2(\R_+,\mu)$ of type $2$. Note however that the
  scaling in $\K$ differs from the scaling in $\B$. Thus,
  Theorem~\ref{thm:1} applies and also \eqref{cond:conv2} in
  Theorem~\ref{thm:2} holds, see Corollary~\ref{cor:banach}.

  In order to construct an example of a regularly varying function in
  $\K'$, take any random function $\eta$ from $\K$ with
  $\norm{\eta}=1$ a.\,s. and define
  \begin{displaymath}
    \xi(x)=(\zeta\cdot\eta)(x)=\eta(\zeta^{-1}x),\qquad x\geq0\,,
  \end{displaymath}
  for a non-negative and independent of $\eta$ random variable $\zeta$
  with regularly varying tail with index $\alpha\geq1$. Then
  $\norm{\xi}=\zeta$ and $\norm{\xi}^{-1}\cdot\xi=\eta$ a.\,s. and
  \begin{align*}
    n\P(\norm{\xi}^{-1}\cdot \xi \in B,\,\norm{\xi}>a_nr) &=
    n\P(\eta \in B)\P(\zeta>a_nr)\to \sigma(B) r^{-\alpha}
  \end{align*}
  as $n\to\infty$, where $a_n$ is the normalising sequence associated
  to $\zeta$, so that $\xi$ is indeed regularly varying in $\K$.

  The condition $\E\norm{\xi}<\infty$ in Theorem~\ref{thm:2} means
  that $\E\zeta<\infty$. We define $A_n(x)=n\E\xi(x)$ for all
  $x$. Condition~\eqref{cond:conv2} holds in case
  $\E\zeta^{\min(\alpha,2)}<\infty$ by Corollary~\ref{cor:banach}.
\end{example}

A number of further examples fall into the scope of the proved large
deviation theorems. They include the cone of compact sets with the
union operation and the Hausdorff metric, and the cone of integrable
probability measures with the convolution operation and rescaling of
the argument and the Wasserstein metric, see \cite[Sec.~8]{DMZ} for
these examples in view of the stability properties.

\section{Acknowledgements}
\label{sec:acknowledgements}

The authors are grateful to Zbyn\v{e}k Pawlas for explaining some
details of the proofs of large deviation results for Minkowski sums of
random sets.  The authors benefited from the advice of Sergei Foss
concerning asymptotic behavior of random walks with heavy-tailed
distributions. 


\end{document}